\documentclass{amsart}
\makeatletter

\@addtoreset{equation}{section}
\makeatother
\usepackage{amsmath,amssymb,amsthm}
\newtheorem{thm}{Theorem}[section]

\newtheorem*{conj1}{The Lindel\"{o}f Hypothesis}
\newtheorem*{conj*}{Conjecture}

\newtheorem{conj}[thm]{Conjecture}
\newtheorem{lem}[thm]{Lemma}
\newtheorem*{thm*}{Theorem}
\newtheorem*{prop*}{Proposition}
\def\f{\frac}
\def\l{\left}
\def\r{\right}
\def\be{\begin{equation}}
\def\ee{\end{equation}}
\def\ba{\begin{align*}}
\def\eal{\end{align*}}
\def\inter{\intertext}
\begin{document}

\title[Uniform upper bounds of relatively r-prime]{Uniform upper bounds of the distribution of relatively r-prime lattice points}
\author{Wataru Takeda}
\address{Department of Mathematics, Kyoto University, Kitashirakawa Oiwake-cho, Sakyo-ku, Kyoto 606-8502,
Japan}
\email{takeda-w@math.kyoto-u.ac.jp}

\begin{abstract}
We estimate the distribution of relatively $r$-prime lattice points in number fields $K$ with their components having a norm less than $x$.  In the previous paper we obtained uniform upper bounds as $K$ runs through all number fields under assuming the Lindel\"of hypothesis. And we also showed unconditional results for abelian extensions with a degree less than or equal to $6$. In this paper we remove all assumption about number fields and improve uniform upper bounds. Throughout this paper we consider estimates for distribution of ideals of the ring of integer $\mathcal{O}_K$ and obtain uniform upper bounds. And when $K$ runs through cubic extension fields we show better uniform upper bounds than that under the Lindel\"  of Hypothesis.
\end{abstract}

\maketitle
\section{Introduction}
Let $K$ be a number field and let $\mathcal{O}_K$ be its ring of integers. 
We regard an $m$-tuple of ideals $(\mathfrak{a}_1, \mathfrak{a}_2,\ldots,\mathfrak{a}_m)$ of $\mathcal{O}_K$ as a lattice point in $K^m$.  We say that a lattice point $(\mathfrak{a}_1, \mathfrak{a}_2,\ldots,\mathfrak{a}_m)$ is {\it relatively $r$-prime} for a positive integer $r$, if there exists no prime ideal $\mathfrak{p}$ such that $\mathfrak{a}_1, \mathfrak{a}_2,\ldots,\mathfrak{a}_m\subset \mathfrak{p}^r$. Let $V^r_m(x,K)$ denote the number of relatively $r$-prime lattice points $(\mathfrak{a}_1, \mathfrak{a}_2,\ldots,\mathfrak{a}_m)$ such that for all $i=1,\ldots,m$ their ideal norm $\mathfrak{Na}_i\le x$. One can show that for all $x\ge1$ and for all number field $K$ the number of relatively 1-prime lattice points $V_1^1(x,K)=1$, so in this paper we assume $rm\ge2$.

There are many results about $V_m^r(x,K)$ from 1800's. In the case $K=\mathbf Q$, integer ideals of $\mathbf Z$ and positive integers are in one-to-one correspondence, so it suffices to consider the distribution of positive integers with special properties. L. Gegenbauer proved that the probability that a positive integer is not divisible by an $r$-th power is $1/\zeta(r)$ for $r\ge2$ in 1885 \cite{ge85} and D. N. Lehmer proved that the probability that $m$ positive integers are relatively prime is $1/\zeta(m)$ in 1900 \cite{Le00}. In 1976 S. J. Benkoski showed that $V_m^r(x,\mathbf Q)\sim x^m/\zeta(rm)$ as a full generalization of these results \cite{Be76}.  

In the general case, B. D. Sittinger dealt with ideals in $\mathcal O_K$ rather than algebraic integers themselves. He showed 
\begin{thm*}[cf. \cite{St10}]
Let $n = [K : \mathbf Q]$ then  
\[V^r_m(x,K)=\frac{c^m}{\zeta_K(rm)}x^m+\left\{
\begin{array}{ll}
O\l(x^{m-\f1n}\r) & \text{ if } m\ge3, \text{ or } m=2 \text{ and } r\ge2,\\
O\l(x^{2-\f1n}\log x\r)& \text{ if }m=2 \text{ and } r=1,\\
O\l(x^{1-\f1n}\log x\r) & \text{ if }m=1 \text{ and }  \frac{n(r-2)}{r-1}=1,\\
O\l(x^{1-\f1n}\r)& \text{ if }m=1 \text{ and } \frac{n(r-2)}{r-1}>1,\\
O\l(x^{\f1r\l(2-\f1n\r)}\r)&\text{ if }m=1 \text{ and } \frac{n(r-2)}{r-1}<1,
\end{array}
\right.\]
where $\zeta_K$ is the Dedekind zeta function over $K$ and $c$ is the residue of $\zeta_K(s)$ at $s=1$. 
\end{thm*}
It is well known that
\begin{equation}
c=\frac{2^{r_1}(2\pi)^{r_2}hR}{w\sqrt{D}},\label{crho}
\end{equation}
where $h$ is the class number of $K$, $r_1$ is the number of real embeddings of $K$, $r_2$ is the number of pairs of
complex embeddings, $R$ is the regulator of $K$, $w$ is the number of roots of unity in $\mathcal{O}^*_K$ and $D$ is absolute value of the discriminant of $K$. This fact (\ref{crho}) is shown from the following fact (\ref{simm}) about the distribution of ideals of $\mathcal O_K$:

Let $I_K(x)$ be the number of ideals of $\mathcal{O}_K$ with their ideal norm less than or equal to $x$. Then 
\begin{equation}
I_K(x)\sim\frac{2^{r_1}(2\pi)^{r_2}hR}{w\sqrt{D}}x\text{ as $x\rightarrow\infty$.}\label{simm}
\end{equation}
 For the proof of this result (\ref{crho}), please see Theorem 5 in Section 8 of \cite{La90}. We denote $\Delta_K(x)$ be the error term of $I_K(x)$, that is, $I_K(x)-cx$.

In the previous paper \cite{Ta16}, we studied some relation between the distribution of relatively $r$-prime lattice points and the Lindel\"{o}f Hypothesis.  This hypothesis is extended for many L-functions. In Iwaniec and Kowalski's book \cite{Iw04}, this hypothesis is written as
\begin{conj1}
Let $L(s,\chi)$ be an $L$-function having an Euler product of degree $d$, then for every $\varepsilon>0$, 
\[
L\l(\f12+it,\chi\r)=O\l(\mathfrak{q}(s,\chi)^\varepsilon\r)\text{ as }\frak q\rightarrow\infty,
\]
where $\mathfrak{q}(s,\chi)$ is the analytic conductor and the constant implied in $O$ depends on $\varepsilon$ alone.  The analytic conductor $\mathfrak{q}(s,\chi)$ is defined as
\[
\mathfrak{q}(s,\chi)=q(\chi)\prod_{j=1}^d(|s+\kappa_j|+3),
\]
where $q(\chi)$ is the conductor of $\chi$ and $\kappa_j$ is the local parameters of $L(s,\chi)$ at infinity.
For the details for the definition of the analytic conductor, one can see Iwaniec and Kowalski's book \cite{Iw04}.
\end{conj1}

In the previous papers, we obtained estimates for the error term
\[E_m^r(x,K)=V_m^r(x,K)-\frac{c^m}{\zeta_K(rm)}x^m\]
as follows:
\begin{thm*}[cf.\cite{Ta16},\cite{Ta17}]
For any fixed number field $K$ with $n = [K : \mathbf Q]$ and for all $\varepsilon>0$
\[E^r_m(x,K)=\left\{
\begin{array}{ll}
O\l(x^{2-\alpha(n)}(\log x)^{2\beta(n)+1}\r)&\text{ if } (m,r)=(2,1),\\
O\l(x^{1-\f{\alpha(n)}2}(\log x)^{2\beta(n)}\r)&\text{ if } (m,r)=(1,2),\\
O\l(x^{m-\alpha(n)}(\log x)^{\beta(n)}\r)&\text{ otherwise }
\end{array}
\right.\]
with
\[\alpha(n)=\left\{
\begin{array}{ll}
\frac2n-\frac8{n(5n+2)}&\text{ if } 3\le n\le6,\\
\frac2n-\frac3{2n^2}&\text{ if } 7\le n\le9\rule[-2mm]{0mm}{6mm},\\
\frac3{n+6}-\varepsilon&\text{ if } n\ge10
\end{array}
\right.\text{ and }
\beta(n)=\left\{
\begin{array}{ll}
\frac{10}{5n+2}&\text{ if } 3\le n\le6,\\
\frac 2n&\text{ if } 7\le n\le9\rule[-2mm]{0mm}{6mm},\\
0&\text{ if } n\ge10.
\end{array}
\right.\]

Moreover if we assume the Lindel\"{o}f Hypothesis for $\zeta_K(s)$, we get \[E_m^r(x,K)=\left\{
\begin{array}{ll}
O\l(x^{\f34+\varepsilon}\r)&\text{ if } (m.r)=(1,2),\\
O\l(x^{m-\f12+\varepsilon}\r)&\text{ otherwise }
\end{array}
\right.\]
for all $\varepsilon>0$.
\end{thm*}

In the case that $[K:\mathbf Q]=3$ this theorem states that the error term \[E_m^r(x,K)=\left\{
\begin{array}{ll}
o\l(x^{\f34+\varepsilon}\r)&\text{ if } (m.r)=(1,2),\\
o\l(x^{m-\f12+\varepsilon}\r)&\text{ otherwise }
\end{array}
\right.\]
for all $\varepsilon>0$, that is, we obtained better result than we assumed the Lindel\"of hypothesis. 

We fixed a number field $K$ in above theorems.  On the other hand, we also showed the following results about $E_m^r(x,K)$ with the number field $K$ being varied by restricting the degree of $K$ to be less than
or equal to $6$. 
\begin{thm*}[cf.\cite{tk17}]
Let $S$ be a subset of $\{K : \text{abelian extension field}~|~[K:\mathbf Q]\le6\}$ and
$n=\max\{[K : \mathbf Q] : K\in S\}$. Then for every $\varepsilon> 0$, we have
\[E_m^r(x,K)=\left\{ 
\begin{array}{lll} 
O\l(x^{\f1r\l(\f{23929}{15960}+\f{89n}{1140}+\varepsilon\r)}D^{\f{31}{95}-\f{m-1}2}\r)& \text{ if }\ rm=2,\\
O\l(x^{m+\f{89n}{1140}-\f{7991}{15960}+\varepsilon}D^{\f{31}{190}-\f{m-1}{2}}\r)&\text{ otherwise }
\end{array}
\right.\] 
as $x,D\rightarrow\infty$, where $\gamma(n)=\frac{7991}{15960}-\frac{89n}{1140}-\varepsilon$ and $K$ runs through elements in $S$ satisfying that $x^{\f1{753}+\varepsilon} > D$.

Moreover if we assume the Lindel\"{o}f Hypothesis for $\zeta_K(s)$, then it holds for all $\varepsilon> 0$ that
\[E_m^r(x,K)=\left\{
\begin{array}{ll}
O\l(x^{\f1r(\f32+\varepsilon)}D^{2\varepsilon-\f{m-1}2}\r)&\text{ if } rm=2,\\
O\l(x^{m-\f12+\varepsilon}D^{\varepsilon-\f{m-1}2}\r)&\text{ otherwise }
\end{array}
\right.\]
as $xD^2\rightarrow\infty$, where $K$ runs through all number fields with $x^{1-2\varepsilon}>D^{1+2\varepsilon}$.
\end{thm*}
In this paper we remove the assumption that $K$ is abelian and the extension degree is small. We only used estimates for Dedekind zeta functions on the critical line in the previous papers \cite{Ta16}, \cite{Ta17} and \cite{tk17}. On the other hand we also consider on other vertical lines.

In Section 2 we introduce that the number of relatively $r$-prime lattice points $V_m^r(x,K)$ can be expressed as the sum of $I_K(x)$ and the relation between an estimate for uniform upper bound of the error term $\Delta_K(x)$ and that of the error term $E_m^r(x,K)$.

In Section 3 when $K$ runs through all number fields, we obtain uniform upper bound from the well-known convexity bound of Dedekind zeta function. Theorem \ref{main2} asserts that
\begin{thm*}
The following estimate holds.\[E_m^r(x,K)=\l\{\begin{array}{ll}
O\l(x^{\frac1r\l(2-\frac n{n+2}+\varepsilon\r)}D^{\frac2{n+2}-\frac {m-1}2+\varepsilon}\right)& \text{ if } rm=2,\\
O\left(x^{m-\frac n{n+2}+\varepsilon}D^{\frac1{n+2}-\frac{m-1}2+\varepsilon}\r)& \text{ otherwise }
\end{array}\r.\]
as $xD\rightarrow\infty$, where $K$ runs through all number fields with $[K : \mathbf Q]\le n$ and $x^{2n} > D^{n+4}$.
\end{thm*}
In this theorem we consider all number fields. On the other hand we also aim to obtain results with some conditions, for example, abelian extensions, Galois extensions or restricting the extension degree of $K$ in the following sections.  In Section 4 we consider the case that $K$ runs through all abelian extension fields. We improve uniform upper bound of $\Delta_K(x)$ for all abelian extension fields $K$ with $n=[K:\mathbf Q]$ for $6\le n\le95$ and obtain 
\begin{thm*}
Let $6\le n\le95$ and for every $\varepsilon> 0$ the following estimate holds.
\[E_m^r(x,K)=\l\{\begin{array}{ll}
O\l(x^{\frac1r\l(2-95c_n+\varepsilon\r)}D^{62c_n-\frac {m-1}2+\varepsilon}\right)& \text{ if } rm=2,\\
O\left(x^{m-95c_n+\varepsilon}D^{31c_n-\frac{m-1}2+\varepsilon}\r)& \text{ otherwise }
\end{array}
\r.\]
 as  $x,D\rightarrow\infty$, where the constant $c_n=\f{2388}{70844n+453093}$ and $K$ runs through all abelian extension fields with $[K:\mathbf{Q}]\le n$ and $x^{\f1{753}+\varepsilon} > D$.
\end{thm*}
In Section 5 we consider the case that $K$ runs through all cubic extension fields. We improve uniform upper bounds of $I_K(x)$ for cubic extension fields with following the W. M\"uller \cite{mu88} way. 
\begin{thm*}
For every $\varepsilon>0$ the following estimate holds.
\[E_m^r(x,K)=\l\{\begin{array}{ll}
O\l(x^{\f1r\l(\frac {139}{96}+\varepsilon\r)}D^{\frac23-\frac {m-1}2}\right)& \text{ if } rm=2,\\
O\left(x^{m-\frac {53}{96}+\varepsilon}D^{\frac13-\frac{m-1}2}\r)& \text{ otherwise }
\end{array}\r.\]
as  $x^{\f{53}{32}}D^{\f32}\rightarrow\infty$,
where $K$ runs through all cubic extension fields with $x^{\f{53}{96}-\varepsilon} > D^{\f56+\varepsilon}$.
\end{thm*}
If $K$ runs through all cubic extension fields with $x^{\f{5}{32}} > D$, then the uniform upper bounds in this theorem is better than that under the Lindel\"  of Hypothesis. Considering with this result, we propose that for all number fields $K$ the best uniform upper bound of the error term $E_m^r(x,K)$ is much better than that under the assumption of the Lindel\"  of Hypothesis in Section 6. Our conjecture is 
\begin{conj*}
For every $\varepsilon> 0$, we have
\[E_m^r(x,K)=\left\{
\begin{array}{ll}
o\l(x^{\f1r(\f32+\varepsilon)}D^{2\varepsilon-\f{m-1}2}\r)&\text{ if } rm=2,\\
o\l(x^{m-\f12+\varepsilon}D^{\varepsilon-\f{m-1}2}\r)&\text{ otherwise }
\end{array}
\right.\]
as $xD^2\rightarrow\infty$, where $K$ runs through all number fields with $x^{1-2\varepsilon}>D^{1+2\varepsilon}$.
\end{conj*}

This conjecture is related to the distribution of ideals of $\mathcal O_K$. The distribution of ideals has been studied for long time and many mathematicians improved upper and lower bounds for the case that $K$ is fixed. But this conjecture states a uniform upper bounds with $K$ being varied. In this paper we only deal with the case that  number field $K$ is varied and aim to obtain uniform upper bounds of $\Delta_K(x)$ to improve $E_m^r(x,K)$. 

\section{The distribution of relatively r-prime lattice points}
In this section we introduce relation between the distribution of relatively $r$-prime lattice points over $K$ and that of ideals of $\mathcal O_K$. The Inclusion--Exclusion Principle shows that 
\begin{eqnarray}
V_m^r(x,K)&=&I_K(x)^m-\sum_{\mathfrak{p}_1}I_K\left(\frac x{\mathfrak{N}\mathfrak{p}_1^r}\right)^m+\sum_{\mathfrak{p}_1,\mathfrak{p}_2}I_K\left(\frac x{\mathfrak{N}\mathfrak{p}_1^r\mathfrak{p}_2^r}\right)^m-\cdots\nonumber\\
&=&\sum_{\mathfrak{N}\mathfrak{a}\le x^{\f1r}}\mu(\mathfrak{a})I_K\left(\frac x{\mathfrak{N}\mathfrak{a}^r}\right)^m. \label{vm}
\end{eqnarray}
where $\mu(\mathfrak{a})$ is the M\"{o}bius function defined as
\[\mu(\mathfrak{a})=\left\{
\begin{array}{ll}
0  &i\!f \  \mathfrak{a}\subset\mathfrak{p}^2 \text{ for some prime ideal } \mathfrak{p},\\
1 &i\!f \ \mathfrak{a}=\mathcal{O}_K,\\
(-1)^s &i\!f \ \mathfrak{a}=\mathfrak{p}_1\cdots \mathfrak{p}_s,\text{ where $\mathfrak{p}_1,\ldots, \mathfrak{p}_s$ are distinct prime ideals.}
\end{array}
\right.\] 
The following lemma will play a crucial role in our computing $V_m^r(x,K)$ by the relation (\ref{vm}).

\begin{lem}
\label{main}
Let $S$ be a subset of the set of all number fields. Let $\alpha_S>0$ and $\beta_S$ be the constants depending on $S$. If we have
\[I_K(x)=c x+O(x^{1-\alpha_S}D^{\beta_S}),\]
 where $K$ runs through all number fields in $S$ satisfying with some condition $C(x,D)$ and $x^{2\alpha_S} > D^{1+2\beta_S}$.
Then we get the following estimate:\[E_m^r(x,K)=\l\{\begin{array}{ll}
O\left(x^{m-\alpha_S}D^{\beta_S-\f{m-1}2}+x^{\f{2-\alpha_S}r}D^{2\beta_S-\f{m-1}2}\right)&\text{ if } \alpha_S\not=\f{mr-2}{r-1} \text{ and }r>1,\\
O\left(x^{m-\alpha_S}\log xD^{\beta_S-\f{m-1}2}+x^{\f{2-\alpha_S}r}D^{\beta_S-\f m2}\right)&\text{ if } \alpha_S=\f{mr-2}{r-1} \text{ and }r>1,\\
O\left(x^{m-\alpha_S}D^{\beta_S-\f{m-1}2}+x^{2-\alpha_S}D^{2\beta_S-\f{m-1}2}\right)&\text{ if } r=1 \text{ and } m>2,\\
O\left(x^{2-\alpha_S}\log xD^{\beta_S-\f{m-1}2}+x^{2-\alpha_S}D^{\beta_S-1}\right)&\text{ if } r=1 \text{ and } m=2,\\
\end{array}
\r.\]
where $K$ runs through elements in $S$ satisfying with $C(x,D)$ and $x^{2\alpha_S} > D^{1+2\beta_S}$.
\end{lem}

\begin{proof}
Equality (\ref{vm}) and  the assumption this lemma lead to
\begin{align*}
V_m^r(x,K)&=\sum_{\mathfrak{N}\mathfrak{a}\le x^{\f1r}}\mu(\mathfrak{a})\left(\frac {cx}{\mathfrak{N}\mathfrak{a}^r}+O\left(\left(\frac x{\mathfrak{N}\mathfrak{a}^r}\right)^{1-\alpha_S}D^{\beta_S}\right)\right)^m.\\
\intertext{By identity (\ref{crho}) and the binomial theorem}
V_m^r(x,K)&=(cx)^m\sum_{\mathfrak{N}\mathfrak{a}\le x^{\f1r}}\frac {\mu(\mathfrak{a})}{\mathfrak{N}\mathfrak{a}^{rm}}+O\left(\sum_{\mathfrak{N}\mathfrak{a}\le x^{\f1r}}\left(\frac {x}{\mathfrak{N}\mathfrak{a}^r}\right)^{m-\alpha_S}D^{\beta_S-\f{m-1}2}\right).
\end{align*}
By using the fact $\displaystyle{\sum_{\mathfrak{a}}\frac {\mu(\mathfrak{a})}{\mathfrak{N}\mathfrak{a}^{rm}}=\frac1{\zeta_K(rm)}}$, we get
\[V_m^r(x,K)=\frac{c^m}{\zeta_K(rm)}x^m-(cx)^m\sum_{\mathfrak{N}\mathfrak{a}> x^{\f1r}}\frac {\mu(\mathfrak{a})}{\mathfrak{N}\mathfrak{a}^{rm}}+O\left(\sum_{\mathfrak{N}\mathfrak{a}\le x^{\f1r}}\left(\frac {x}{\mathfrak{N}\mathfrak{a}^r}\right)^{m-\alpha_S}D^{\beta_S-\f{m-1}2}\right).\]

\begin{align*}
\intertext{Now we estimate the behavior of first sum. From the assumption of this theorem we can estimate $I_K(x)-I_K(x-1)=O(x^{1-\alpha_S}D^{\beta_S})$, so we have}
(cx)^m\sum_{\mathfrak{N}\mathfrak{a}> x^{\f1r}}\frac {\mu(\mathfrak{a})}{\mathfrak{N}\mathfrak{a}^{rm}}&=O\left(x^m\int_{x^{\f1r}}^{\infty}\frac{y^{1-\alpha_S}D^{\beta_S-\f m2}}{y^{rm}}\ dy\right)\\
&=O\l(x^{\f{2-\alpha_S}r}D^{\beta_S-\f m2}\r).
\end{align*}
\begin{align*}
\intertext{Next we deal the second sum. As well as first sum, it holds that $I_K(x)-I_K(x-1)=O(x^{1-\alpha_S}D^{\beta_S})$, so we have}
&\sum_{\mathfrak{N}\mathfrak{a}\le x^{\f1r}}\left(\frac {x}{\mathfrak{N}\mathfrak{a}^r}\right)^{m-\alpha_S}D^{\beta_S-\f{m-1}2}\\
=&O\left(x^{m-\alpha_S}D^{\beta_S-\f{m-1}2}\left(1+\int_1^{x^{\f1r}}\frac{y^{1-\alpha_S}D^{\beta_S}}{y^{r(m-\alpha_S)}}\ dy\right)\right)\\
=&\l\{\begin{array}{ll}
O\left(x^{m-\alpha_S}D^{\beta_S-\f{m-1}2}+x^{\f{2-\alpha_S}r}D^{2\beta_S-\f{m-1}2}\right)&\text{ if } \alpha_S\not=\f{mr-2}{r-1} \text{ and }r>1,\\
O\left(x^{m-\alpha_S}\log xD^{\beta_S-\f{m-1}2}\right)&\text{ if } \alpha_S=\f{mr-2}{r-1} \text{ and }r>1,\\
O\left(x^{m-\alpha_S}D^{\beta_S-\f{m-1}2}+x^{2-\alpha_S}D^{2\beta_S-\f{m-1}2}\right)&\text{ if } r=1 \text{ and } m>2,\\
O\left(x^{2-\alpha_S}\log xD^{\beta_S-\f{m-1}2}\right)&\text{ if } r=1 \text{ and } m=2.\\
\end{array}
\r.
\inter{Hence we obtain
\[E_m^r(x,K)=\l\{\begin{array}{ll}
O\left(x^{m-\alpha_S}D^{\beta_S-\f{m-1}2}+x^{\f{2-\alpha_S}r}D^{2\beta_S-\f{m-1}2}\right)&\text{ if } \alpha_S\not=\f{mr-2}{r-1} \text{ and }r>1,\\
O\left(x^{m-\alpha_S}\log xD^{\beta_S-\f{m-1}2}+x^{\f{2-\alpha_S}r}D^{\beta_S-\f m2}\right)&\text{ if } \alpha_S=\f{mr-2}{r-1} \text{ and }r>1,\\
O\left(x^{m-\alpha_S}D^{\beta_S-\f{m-1}2}+x^{2-\alpha_S}D^{2\beta_S-\f{m-1}2}\right)&\text{ if } r=1 \text{ and } m>2,\\
O\left(x^{2-\alpha_S}\log xD^{\beta_S-\f{m-1}2}+x^{2-\alpha_S}D^{\beta_S-1}\right)&\text{ if } r=1 \text{ and } m=2,\\
\end{array}
\r.\]
where $K$ runs through all number fields in $S$ satisfying with the condition $C(x,D)$ and $x^{2\alpha_S} > D^{1+2\beta_S}$.
}
\inter{This proves the Lemma.}
\end{align*}
\end{proof}
From this lemma it is very important to obtain good uniform upper bounds for the distribution of ideals of $\mathcal O_K$. In the following sections we consider the uniform upper bounds of $\Delta_K(x)$.

\section{Uniform bound}
Let $s=\sigma+it$ and $n=[K:\mathbf{Q}]$. In this section we use the convexity bound of the Dedekind zeta function to obtain a uniform upper bound of the distribution of relatively $r$-prime lattice points. It is obtained from the Phragmen-Lindel\"of principle, which is very important classical facts of complex analysis. 
\begin{lem}[Phragmen-Lindel\"of principle]
\label{plp}
Let $f$ be a function holomorphic on an open neighborhood of a strip $a\le\sigma\le b$ and $|f(s)|=O\l(exp(|s|^{\alpha})\r)$ for some $\alpha\ge1$ and $a\le\sigma\le b$. Then the followings hold.
\begin{enumerate}
\item Assume that $|f(s)|=O\l(t^M\r)$ for all $s$ on the boundary of the strip. Then we have $|f(s)|=O\l(t^M\r)$ for all $s$ in the strip.
\item Assume that $|f(a+it)|\le M_a|t|^{\alpha}$ and $|f(b+it)|\le M_b|t|^{\beta}$ for $|t|\ge 1$.
Then\[f(\sigma+it)\le M_a^{\ell(\sigma)}M_b^{1-\ell(\sigma)}|t|^{\alpha\ell(\sigma)+\beta(1-\ell(\sigma))}\]
for all $s$ in the strip, where $\ell$ is the linear function such that $\ell(a)=1$ and $\ell(b)=0$, that is,\[\ell(\sigma)=\f1{a-b}\sigma-\f b{a-b}.\]
\end{enumerate}
\end{lem}

It is well-known fact that Dedekind zeta function satisfies the following functional equation 
\begin{equation}
\label{fe}
Z_K(1-s)=Z_K(s), 
\end{equation}
$\displaystyle \text{ where } Z(s)=D^{\f s2}2^{-(s-1)r_2}\pi^{-\f{ns}2}\Gamma\l(\f s2\r)^{r_1}\Gamma(s)^{r_2}\zeta_K(s).$

For $\sigma>1$ the Dedekind zeta function satisfy that $|\zeta_K(\sigma+it)|\le\zeta_K(\sigma)$ and $\zeta_K(\sigma)$ converges, so the growth rate of the Dedekind zeta function $|\zeta_K(\sigma+it)|\le \l(|t|^nD\r)^{\varepsilon}$ for all $\varepsilon>0$ and for all $\sigma>1$. 
This upper bound of Dedekind zeta function  for $\sigma>1$, the functional equation (\ref{fe}) and the Stirling's estimate for the gamma function give an estimate $|\zeta_K(\sigma+it)|\le \l(|t|^nD\r)^{\f12-\sigma+\varepsilon}$ for $\sigma\le0$.

From the Phragmen-Lindel\"of principle one can obtain the well-known convexity bound of the Dedekind zeta function on $0\le\sigma\le1$:
\[
\zeta_K(s)=O\l(|t|^{n\frac{1-\sigma}2+\varepsilon}D^{\frac{1-\sigma}2+\varepsilon}\r)
\]
as $|t|^nD\rightarrow\infty$, where the constant implied in $O$ depends on $\varepsilon$ alone. This uniform bound can be summarized by the formulas:

For all $\varepsilon>0$ and $n=[K:\mathbf Q]$
\be\label{convex}
\zeta_K(\sigma+it)=\l\{\begin{array}{ll}
O\l(|t|^{\f n2-n\sigma+\varepsilon}D^{\f 12-\sigma+\varepsilon}\r)& \text{ if }\sigma\le0,\\
O\l(|t|^{n\frac{1-\sigma}2+\varepsilon}D^{\frac{1-\sigma}2+\varepsilon}\r)&\text{ if }0\le\sigma\le1,\\
O\l(|t|^{\varepsilon}D^{\varepsilon}\r)&\text{ if } 1\le\sigma
\end{array}\r.
\ee
as $|t|^nD\rightarrow\infty$, where $K$ runs through all number fields with $[K:\mathbf{Q}]=n$. In the precious papers, we used upper bound of Dedekind zeta function to estimate the distribution of ideals. Following the way in our precious papers, this uniform convexity bound for Dedekind zeta functions leads the following theorem.
\begin{thm}
\label{num}
The following estimate holds. For all $\varepsilon>0$
\[I_K(x)=c x+O(x^{1-\frac n{n+2}+\varepsilon}D^{\frac1{n+2}+\varepsilon})\]
as  $xD\rightarrow\infty$, where $K$ runs through all number fields with $[K:\mathbf{Q}]\le n$ and $x^{2n} > D^{n+4}$.
\end{thm}
\begin{proof}
It suffices to consider for all half integer $x=k+\f12$, where $k$ is a positive integer, because it holds for any real number $y\in[k,k+1)$ that $I_K(x)=I_K(y)$. We consider the integral  
\[\frac1{2\pi i}\int_{C}\zeta_K(s)\frac{x^s}s\ ds,\]
where $C$ is the contour $C_1\cup C_2\cup C_3\cup C_4$ in the following figure.

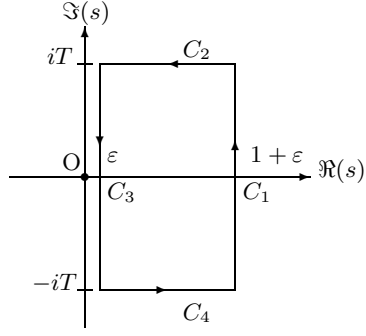
\setlength\unitlength{1truecm}
\[\begin{picture}(4.5,4.5)(0,0)
\small
\put(-1,2){\vector(1,0){4}}
\put(0,0){\vector(0,1){4}}
\put(0.2,0.5){\vector(1,0){0.9}}
\put(1.1,0,5){\line(1,0){0.9}}
\put(2,0.5){\vector(0,1){2}}
\put(2,2.4){\line(0,1){1.1}}
\put(2,3.5){\vector(-1,0){0.9}}
\put(1.1,3,5){\line(-1,0){0.9}}
\put(-0.3,2.1){O}
\put(3.1,2){$\Re(s)$}
\put(-0.3,4.1){$\Im(s)$}
\put(0,2){\circle*{0.1}}
\put(2,3.5){\line(-1,0){0.5}}
\put(0.2,3.5){\vector(0,-1){1.1}}
\put(0.2,2.5){\line(0,-1){2}}
\put(-0.1,3.5){\line(1,0){0.2}}
\put(-0.5,3.5){$iT$}
\put(-0.1,0.5){\line(1,0){0.2}}
\put(-0.7,0.5){$-iT$}
\put(2,2.9){\line(0,1){0.2}}
\put(0.3,2.2){$\varepsilon$}
\put(2.2,2.2){$1+\varepsilon$}
\put(1.3,3.6){$C_2$}
\put(2.1,1.7){$C_1$}
\put(1.3,0.1){$C_4$}
\put(0.3,1.7){$C_3$}
\end{picture}\]

In a way similar to the well-known proof of Perron's formula, we estimate
\[\frac1{2\pi i}\int_{C_1}\zeta_K(s)\frac{x^s}s\ ds=I_K(x)+O\left(\frac{x^{1+\varepsilon}}{T}\right).\]
We can select the large $T$, so that the $O$-term in the right hand side is sufficiently small. For estimating the left hand side by using estimate (\ref{convex}), we divide it into the integrals over $C_2, C_3$ and $C_4$. 
\begin{align*}
\intertext{First we consider the integral over $C_3$ as}
\left|\frac1{2\pi i}\int_{C_3}\zeta_K(s)\frac{x^s}s\ ds\right|&=\left|\frac1{2\pi i}\int_{T}^{-T}\zeta_K\left(\varepsilon+it\right)\frac{x^{\varepsilon+it}}{\varepsilon+it}i\ dt\right|.\\
\intertext{From estimate (\ref{convex}), it holds that}
\left|\frac1{2\pi i}\int_{C_3}\zeta_K(s)\frac{x^s}s\ ds\right|&=O\left(\int^{T}_{-T}T^{\f n2(1-\varepsilon)+\varepsilon}D^{\f12(1-\varepsilon)+\varepsilon}\frac {x^{\varepsilon}}{\left|\varepsilon+it\right|}\ dt\right)\\
&=O(x^{\varepsilon}T^{\f n2+\varepsilon}D^{\f12+\varepsilon}).
\end{align*}

\begin{align*}
\intertext{Next we calculate the integral over $C_2$ and $C_4$ as}
\left|\frac1{2\pi i}\int_{C_2+C_4}\zeta_K(s)\frac{x^s}s\ ds\right|&\le\frac1{2\pi}\int^{1+\varepsilon}_{\varepsilon}\left|\zeta_K\left(\sigma+iT\right)\right|\frac{x^{\sigma}}{\left|\sigma+iT\right|}\ d\sigma.\\
\intertext{From estimate (\ref{convex}), it holds that}
\left|\frac1{2\pi i}\int_{C_2+C_4}\zeta_K(s)\frac{x^s}s\ ds\right|&=O\left(\int^{1+\varepsilon}_{\varepsilon}(T^nD)^{\max\{\frac{1-\sigma}2,0\}+\varepsilon}\frac{x^{\sigma}}{T}\ d\sigma\r)\\
&=O\left(\int^{1}_{\varepsilon}T^{\f n2+\varepsilon}D^{\f12+\varepsilon}\l(\frac{x^2}{T^nD}\r)^{\f \sigma2}\ d\sigma+\frac{x^{1+\varepsilon}T^{\varepsilon}D^{\varepsilon}}{T}\right).\\
\intertext{When the inequality $x^2<T^nD$ holds, the function $\l(\frac{x^2}{T^nD}\r)^{\f \sigma2}\le\l(\frac{x^2}{T^nD}\r)^{\f \varepsilon2}$ for $\sigma\in[\varepsilon,1]$, so we estimate}
\left|\frac1{2\pi i}\int_{C_2+C_4}\zeta_K(s)\frac{x^s}s\ ds\right|&=O\left(\frac{x^{\varepsilon}T^{\f n2+\varepsilon}D^{\f12+\varepsilon}}{T}\right)+O\l(\frac{x^{1+\varepsilon}T^{\varepsilon}D^{\varepsilon}}{T}\right).
\end{align*}
By Cauchy's residue theorem we get \[\frac1{2\pi i}\int_{C}\zeta_K(s)\frac{x^s}s\ ds=c x.\]
\begin{align*}
\intertext{By using all result above, it is obtained that}
I_K(x)&=cx+O(x^{\varepsilon}T^{\f n2+\varepsilon}D^{\f12+\varepsilon})+O\left(\frac{x^{\varepsilon}T^{\f n2+\varepsilon}D^{\f12+\varepsilon}}{T}\right)+O\l(\frac{x^{1+\varepsilon}T^{\varepsilon}D^{\varepsilon}}{T}\right).\\
\intertext{When we select $T=x^{\frac 2{n+2}}D^{-\frac1{n+2}}$, this becomes}
I_K(x)&=c x+O\l(x^{1-\frac 2{n+2}+\varepsilon}D^{\frac1{n+2}+\varepsilon}\r).
\inter{This proves this theorem.}
\end{align*}
\end{proof}

When $K$ runs through all number fields we obtain a uniform upper bound on the distribution of relatively $r$-prime lattice points. When $K$ runs through all number fields with $[K : \mathbf Q]\le n$, the two constants $\alpha_S$ and $\beta_S$ in Lemma \ref{main} are $\f2{n+2}-\varepsilon$ and $\f1{n+2}-\varepsilon$ respectively. Lemma \ref{main} and Theorem \ref{num} lead to

\begin{thm}
\label{main2}
The following estimate holds. For all $\varepsilon>0$
\[E_m^r(x,K)=\l\{\begin{array}{ll}
O\l(x^{(2-\frac n{n+2}+\varepsilon)\frac1r}D^{\frac2{n+2}-\frac {m-1}2+\varepsilon}\right)& \text{ if } rm=2,\\
O\left(x^{m-\frac n{n+2}+\varepsilon}D^{\frac1{n+2}-\frac{m-1}2+\varepsilon}\r)& \text{ otherwise }
\end{array}\r.\]
as $xD\rightarrow\infty$, where $K$ runs through all number fields with $[K : \mathbf Q]\le n$ and $x^{2n} > D^{n+4}$.
\end{thm}

We consider all number fields in this section, but in this paper we also aim to obtain results with some condition, for example, abelian extensions, Galois extensions or restricting the extension degree of $K$. In the following sections we consider some special cases.


\section{Abelian extension case}
In this section we assume $K$ is an abelian extension field.
The case that $K$ runs through all abelian extension field of $\mathbf Q$. An estimate similar to Theorem \ref{num} holds for this case. 
When $K$ is an abelian extension field, we have a factorization of the Dedekind zeta function as \[\zeta_K(s)=\zeta_{\mathbf{Q}}(s)\prod_{\chi}L(s,\chi),\]
where the Dirichlet character $\chi$ runs through Dirichlet characters
so that the product of their conductors is equal to $D$.
In this section we use some result for estimate for Riemann zeta function and Dirichlet L-functions on the critical line under a certain condition.

Huxley and Watt considered the uniform upper bound of $L(s,\chi)$ on the critical line and showed that for primitive Dirichlet characters $\chi$ modulo $p$
\begin{equation}\label{HW}
L\left(\frac12+it,\chi\right)=O\left(|t|^{\f{89}{570}+\varepsilon}p^{\f{31}{190}}\right)
\end{equation}
as $|t|,p\rightarrow\infty$ with $p<|t|^{\frac2{753}}$ for all $\varepsilon>0$ \cite{HW00}.
Bourgain showed that
\begin{equation}\label{Bourgain}
\zeta_{\mathbf{Q}}\l(\f12+it\r)=O\left(|t|^{\f{13}{84}+\varepsilon}\right)
\end{equation}
as $|t|\rightarrow\infty$ for all $\varepsilon>0$ \cite{Bo17}.

Since the Dedekind zeta function $\zeta_K(s)$ has an Euler product of degree $n=[K:\mathbf{Q}]$, 
it holds from \eqref{HW} and \eqref{Bourgain} that for every $\varepsilon>0$
\begin{equation}
\zeta_K\left(\frac12+it\right)=O\left(|t|^{\f{89n}{570}-\f{11}{7980}+\varepsilon}D^{\f{31}{190}}\right)\label{23}
\end{equation}
as $|t|,D\rightarrow\infty$, where $K$ runs through all abelian extension fields with $[K:\mathbf{Q}]\le n$ and $D<|t|^{\f2{753}}$.

The Phragmen-Lindel\"of principle (\ref{plp}) and estimate (\ref{23}) lead to the following estimate about Dedekind zeta functions of $K$:

\begin{equation}
\label{dede}
\zeta_K(\sigma+it)=\l\{\begin{array}{ll}
O\l(|t|^{\f n2-n\sigma+\varepsilon}D^{\f 12-\sigma+\varepsilon}\r)& \text{ if }\sigma\le0,\\
O\l(|t|^{\f n2-\f{196}{285}n\sigma-\f{11}{3980}\sigma+\varepsilon}D^{\f12-\f{64}{95}\sigma+\varepsilon}\r)& i\!f\ 0\le\sigma\le\f12,\\ 
O\l(|t|^{{{89}\over{285}}n-{{89}\over{285}}n\sigma+{{11}\over{3990}}\sigma-{{11}\over{3990}}+\varepsilon}D^{\f{31}{95}-\f{31}{95}\sigma+\varepsilon}\r)& i\!f\ \f12\le\sigma\le1,\\
O\l(|t|^{\varepsilon}D^{\varepsilon}\r)&\text{ if } 1\le\sigma
\end{array}
\r. 
\end{equation}
as $|t|,D\rightarrow\infty$, where $K$ runs through all abelian extension fields with $[K:\mathbf{Q}]\le n$ and $D<|t|^{\f2{753}}$.

This estimate states better uniform bounds than the uniform upper convexity bounds (\ref{convex}), so we expect a good result.
\begin{thm}
\label{abel}
The following estimate holds. For all $\varepsilon>0$
\[I_K(x)=c x+\l\{\begin{array}{ll}
O\l(x^{1-\frac 2{n+2}+\varepsilon}D^{\frac1{n+2}+\varepsilon}\r)& \text{ if  $n\le5$ or $n\ge96$},\\
O\l(x^{1-95c_n+\varepsilon}D^{31c_n+\varepsilon}\r)& \text{ if  $6\le n\le95$}
\end{array}
\r.\]
 as  $x,D\rightarrow\infty$, where $c_n=\f{2388}{70844n+453093}$ and $K$ runs through all abelian extension fields with $[K:\mathbf{Q}]\le n$ and $x^{\f1{753}+\varepsilon} > D$.
\end{thm}
\begin{proof}
We can show this theorem in a way similar to the proof of Theorem \ref{num}.
We consider similar integral \[\frac1{2\pi i}\int_{C}\zeta_K(s)\frac{x^s}s\ ds,\]
where $C$ is the contour $C_1\cup C_2\cup C_3\cup C_4$ in the following figure. Let $\delta$ be a positive constant with $0<\delta\le\f12$,

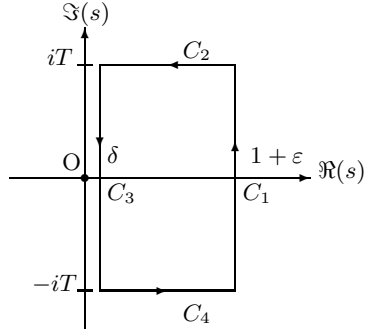
\setlength\unitlength{1truecm}
\[\begin{picture}(4.5,4.5)(0,0)
\small
\put(-1,2){\vector(1,0){4}}
\put(0,0){\vector(0,1){4}}
\put(0.2,0.5){\vector(1,0){0.9}}
\put(1.1,0,5){\line(1,0){0.9}}
\put(2,0.5){\vector(0,1){2}}
\put(2,2.4){\line(0,1){1.1}}
\put(2,3.5){\vector(-1,0){0.9}}
\put(1.1,3,5){\line(-1,0){0.9}}
\put(-0.3,2.1){O}
\put(3.1,2){$\Re(s)$}
\put(-0.3,4.1){$\Im(s)$}
\put(0,2){\circle*{0.1}}
\put(2,3.5){\line(-1,0){0.5}}
\put(0.2,3.5){\vector(0,-1){1.1}}
\put(0.2,2.5){\line(0,-1){2}}
\put(-0.1,3.5){\line(1,0){0.2}}
\put(-0.5,3.5){$iT$}
\put(-0.1,0.5){\line(1,0){0.2}}
\put(-0.7,0.5){$-iT$}
\put(2,2.9){\line(0,1){0.2}}
\put(0.3,2.2){$\delta$}
\put(2.2,2.2){$1+\varepsilon$}
\put(1.3,3.6){$C_2$}
\put(2.1,1.7){$C_1$}
\put(1.3,0.1){$C_4$}
\put(0.3,1.7){$C_3$}
\end{picture}\]

From Estimate (\ref{dede}) and a way similar to the proof of Theorem \ref{num} it holds that
\begin{align*}
&\frac1{2\pi i}\int_{C_i}\zeta_K(s)\frac{x^s}s\ ds\\
=&\l\{\begin{array}{ll}
I_K(x)+O\left(\f{x^{1+\varepsilon}}T\right) & \text{ if }\ i=1,\\
O\left(x^{\delta}T^{\f n2-\f{196}{285}n\delta-\f{11}{3980}\delta-1+\varepsilon}D^{\f12-\f{64}{95}\delta+\varepsilon}\right)+O\l(\f{x^{1+\varepsilon}T^{\varepsilon}D^{\varepsilon}}T\right)& \text{ if } i=2,4,\\
O\left(x^{\delta}T^{\f n2-\f{196}{285}n\delta-\f{11}{3980}\delta+\varepsilon}D^{\f12-\f{64}{95}\delta+\varepsilon}\right) & \text{ if }\ i=3.
\end{array}
\r.
\end{align*}
By Cauchy's residue theorem we get \[\frac1{2\pi i}\int_{C}\zeta_K(s)\frac{x^s}s\ ds=c x.\]
\begin{align*}
\intertext{By using all result above, it is obtained that}
I_K(x)&=cx+O\left(x^{\delta}T^{\f n2-\f{196}{285}n\delta-\f{11}{3980}\delta+\varepsilon}D^{\f12-\f{64}{95}\delta+\varepsilon}\right)+O\l(\f{x^{1+\varepsilon}T^{\varepsilon}D^{\varepsilon}}T\right).\\
\intertext{When we select \[T=x^{190(1-\delta)c(n,\delta)}D^{(128\delta-95)c(n,\delta)},\] where $c(n,\delta)=\f{1194}{226860+113430n-627\delta-156016n\delta}$, this becomes}
I_K(x)&=c x+O(x^{1-190(1-\delta)c(n,\delta)+\varepsilon}D^{(95-128\delta)c(n,\delta)+\varepsilon}).
\end{align*}
For $\delta\in\l(0,\f12\r]$ two functions $1-190(1-\delta)c(n,\delta)$ and $(95-128\delta)c(n,\delta)$ is monotone functions, thus 
\begin{align*}
I_K(x)&=c x+O(x^{1-\frac 2{n+2}+\varepsilon}D^{\frac1{n+2}+\varepsilon})+O(x^{1-95c_n+\varepsilon}D^{31c_n+\varepsilon}),\\
\intertext{where the constant $c_n=\f{2388}{70844n+453093}$. Considering the two $O$-estimates, we can conclude that}
I_K(x)&=c x+\l\{\begin{array}{ll}
O(x^{1-\frac 2{n+2}+\varepsilon}D^{\frac1{n+2}+\varepsilon})& \text{ if  $n\le5$ or $n\ge96$},\\
O(x^{1-95c_n+\varepsilon}D^{31c_n+\varepsilon})& \text{ if  $6\le n\le95$}.
\end{array}
\r.
\inter{This proves this theorem.}
\end{align*}
\end{proof}
When $6\le n:=[K:\mathbf Q]\le95$ and $K$ runs through only abelian extension fields, our result $I_K(x)=c x+O(x^{1-95c_n+\varepsilon}D^{31c_n+\varepsilon})$ is much better than result given by the convexity bound. Lemma \ref{main} and Theorem \ref{abel} lead to 
\begin{thm}
\label{abelian}
Let $6\le n\le95$ and for every $\varepsilon> 0$ the following estimate holds.
\[E_m^r(x,K)=\l\{\begin{array}{ll}
O\l(x^{\frac1r\l(2-95c_n+\varepsilon\r)}D^{62c_n-\frac {m-1}2+\varepsilon}\right)& \text{ if } rm=2,\\
O\left(x^{m-95c_n+\varepsilon}D^{31c_n-\frac{m-1}2+\varepsilon}\r)& \text{ otherwise }
\end{array}
\r.\]
 as  $x,D\rightarrow\infty$, where the constant $c_n=\f{2388}{70844n+453093}$ and $K$ runs through all abelian extension fields with $[K:\mathbf{Q}]\le n$ and $x^{\f1{753}+\varepsilon} > D$.
\end{thm}

Theorem \ref{abelian} states better upper uniform upper bounds than our previous results \cite{tk17}. Moreover we can remove the assumption $[K:\mathbf Q]\le6$. 

\section{The case of cubic extensions}
The case that $K$ is fixed cubic extension field, it is known that the error term $E_m^r(x,K)=o(x^{\f12})$, that is, better result than results under assuming the Lindel\"of hypothesis is obtained \cite{Ta17}. In this section, we consider the distribution of ideals of $\mathcal{O}_K$, where $K$ runs through all cubic extensions. W. M\"uller showed that for a fixed cubic extension $K$ and for all $\varepsilon>0$ 
\begin{equation}
\label{cub}
I_K(x)=cx+O\l(x^{\f{43}{96}+\varepsilon}\r)
\end{equation}
as $x\rightarrow\infty$ \cite{mu88}. As a generalization of M\"uller's result, we estimate uniform upper bound of $I_K(x)$. 

\begin{thm}
\label{cubic}
For every $\varepsilon>0$ the following estimate holds.
\[I_K(x)=c x+O\l(x^{\f{43}{96}+\varepsilon}D^{\f13}\r)\]
as  $x^{\f{53}{32}}D^{\f32}\rightarrow\infty$, where $K$ runs through all cubic extension fields with $x^{\f{53}{96}-\varepsilon} > D^{\f56}$.
\end{thm}
\begin{proof}
We estimate the distribution of ideals $I_K(x)$ by following the W. M\"uller way.

We consider the integral \[\frac1{2\pi i}\int_{C}\zeta_K(s)\frac{x^s}s\ ds,\]
where $C$ is the contour $C_1\cup C_2\cup C_3\cup C_4$ in the following figure.

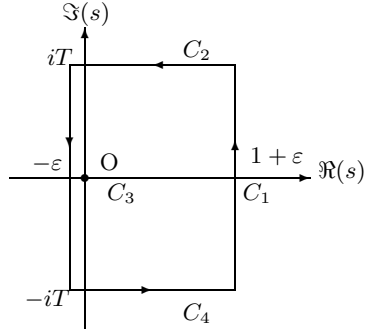
\setlength\unitlength{1truecm}
\[\begin{picture}(4.5,4.5)(0,0)
\small
\put(-1,2){\vector(1,0){4}}
\put(0,0){\vector(0,1){4}}
\put(-0.2,0.5){\vector(1,0){1.1}}
\put(0.9,0,5){\line(1,0){1.1}}
\put(2,0.5){\vector(0,1){2}}
\put(2,2.4){\line(0,1){1.1}}
\put(2,3.5){\vector(-1,0){1.1}}
\put(0.9,3,5){\line(-1,0){1.1}}
\put(0.2,2.1){O}
\put(3.1,2){$\Re(s)$}
\put(-0.3,4.1){$\Im(s)$}
\put(0,2){\circle*{0.1}}
\put(2,3.5){\line(-1,0){0.5}}
\put(-0.2,3.5){\vector(0,-1){1.1}}
\put(-0.2,2.5){\line(0,-1){2}}
\put(-0.1,3.5){\line(1,0){0.2}}
\put(-0.5,3.5){$iT$}
\put(-0.1,0.5){\line(1,0){0.2}}
\put(-0.8,0.3){$-iT$}
\put(2,2.9){\line(0,1){0.2}}
\put(-0.7,2.1){$-\varepsilon$}
\put(2.2,2.2){$1+\varepsilon$}
\put(1.3,3.6){$C_2$}
\put(2.1,1.7){$C_1$}
\put(1.3,0.1){$C_4$}
\put(0.3,1.7){$C_3$}
\end{picture}\]

\begin{align*}
\intertext{We estimate the integral over $C_1$, $C_2$ and $C_4$ in the same way to before. In this proof we only consider the integral over $C_3$ as}
\frac1{2\pi i}\int_{C_3}\zeta_K(s)\frac{x^s}s\ ds&=\frac 1{2\pi i}\int_{-\varepsilon-iT}^{-\varepsilon+iT}\zeta_K(s)\frac{x^{s}}{s}\ ds.\\
\inter{Changing the variable $s$ to $1-s$, we have}
\frac1{2\pi i}\int_{C_3}\zeta_K(s)\frac{x^s}s\ ds&=\frac 1{2\pi i}\int_{1+\varepsilon-iT}^{1+\varepsilon+iT}\zeta_K(1-s)\frac{x^{1-s}}{1-s}\ ds.\\
\end{align*}
\begin{align*}
\intertext{From the functional equation (\ref{fe}), it holds that}
&\frac1{2\pi i}\int_{C_3}\zeta_K(s)\frac{x^s}s\ ds\\
=&\frac 1{2\pi i}\int_{1+\varepsilon-iT}^{1+\varepsilon+iT}D^{s-\f12}2^{3(1-s)}\pi^{-3s}\Gamma(s)^{3}\cos \l(\f{\pi s}2\r)^{r_1+r_2}\sin \l(\f{\pi s}2\r)^{r_2}\zeta_K(s)\frac{x^{1-s}}{1-s}\ ds.\\
\end{align*} 
Now we consider the behavior of $\Gamma(s)^3(1-s)^{-1}$. From the property of gamma function: $\Gamma(s)=(s-1)\Gamma(s-1)$, we have
\begin{align*}
\frac{\Gamma(s)^3}{1-s}&=-\Gamma(s)^2\Gamma(s-1).\\
\intertext{By the Stirling formula: 
\begin{equation}\label{stir}
\Gamma(s)=\sqrt{\f{2\pi}{s}}\l(\f se\r)^s(1+O(s^{-1})),
\end{equation} 
we estimate}
\frac{\Gamma(s)^3}{1-s}&=-\f{2\pi}{s}\l(\f se\r)^{2s}\sqrt{\f{2\pi}{s-1}}\l(\f {s-1}e\r)^{s-1}(1+O(s^{-1}))\\
&=-3^{\f52-3s}e2\pi\l(\f {3s-2}e\r)^{3s-2}\sqrt{\f{2\pi}{3s-2}}(1+O(s^{-1})).\\
\end{align*}
We use the Stirling formula (\ref{stir}) for $s=1+\varepsilon+it$ and obtain
\begin{equation}\label{gamma}
\frac{\Gamma(s)^3}{1-s}=C3^{-3s}\Gamma(3s-2)+O(|t|^{-\f12+3\varepsilon}e^{-\f32\pi|t|}),
\end{equation}
where $C$ is a constant.

Next we estimate the behavior of $\cos \l(\f{\pi s}2\r)^{r_1+r_2}\sin \l(\f{\pi s}2\r)^{r_2}$. From the property of $\sin$ and $\cos$, we have
\begin{align*}
\cos \l(\f{\pi s}2\r)^{r_1+r_2}\sin \l(\f{\pi s}2\r)^{r_2}&=\l\{\begin{array}{ll}
\sin \l(\f{\pi s}2\r)-\sin \l(\f{\pi s}2\r)^{3}&\text{ if }r_1=r_2=1,\\
\cos \l(\f{\pi s}2\r)^{3}& \text{ if }r_1=3,
\end{array}\r.\\
&=\l\{\begin{array}{ll}
\f14\sin \l(\f{3\pi s}2\r)+\f14\sin \l(\f{\pi s}2\r)&\text{ if }r_1=r_2=1,\\
\f14\cos \l(\f{3\pi s}2\r)+\f34\cos \l(\f{\pi s}2\r)& \text{ if }r_1=3.
\end{array}\r.\\
\end{align*}
We use the well-known estimates $\cos (\sigma+it)=O(e^{|t|})$ and  $\sin (\sigma+it)=O(e^{|t|})$ for $s=1+\varepsilon+it$ and get
\be\label{sincos}
\cos \l(\f{\pi s}2\r)^{r_1+r_2}\sin \l(\f{\pi s}2\r)^{r_2}=C f\l(\f{3\pi s}2\r)+O\left(e^{\f{|t|\pi}2}\r),\\
\ee
where $C$ is a constant and $f$ is $\cos$ or $\sin$.

Estimate (\ref{gamma}) and (\ref{sincos}) lead
\be\label{gcs}
\f{\Gamma(s)^{3}}{1-s}\cos \l(\f{\pi s}2\r)^{r_1+r_2}\sin \l(\f{\pi s}2\r)^{r_2}=C\Gamma(3s-2)f\l(\f{3\pi s}2\r)+O\l(|t|^{-\f12+3\varepsilon}\r).
\ee

\begin{align*}
\inter{By Estimate (\ref{gcs}) the integral over $C_3$ can be expressed as}
&\frac1{2\pi i}\int_{C_3}\zeta_K(s)\frac{x^s}s\ ds\\
=&\frac C{2\pi i}\int_{1+\varepsilon-iT}^{1+\varepsilon+iT}D^{s-\f12}6^{-3s}\pi^{-3s}\Gamma(3s-2)f\l(\f{3\pi s}2\r)\zeta_K(s)x^{1-s}\ ds\\
&+O\l(\int_{-T}^{T}D^{\f12+\varepsilon}|t|^{-\f12+\varepsilon}\zeta_K(1+\varepsilon+it)x^{-\varepsilon}\ dt\r).
\inter{From the well-known estimate  $|\zeta_K(1+\varepsilon+it)|\le \l(|t|D\r)^{\varepsilon}$, }
&\frac1{2\pi i}\int_{C_3}\zeta_K(s)\frac{x^s}s\ ds\\
=&\frac {Cx}{2\pi i}\int_{1+\varepsilon-iT}^{1+\varepsilon+iT}D^{-\f12}\l(\f{6^{3}\pi^{3}x}{D}\r)^{-s}\Gamma(3s-2)f\l(\f{3\pi s}2\r)\zeta_K(s)\ ds\\
&+O\l(\int_{-T}^{T}D^{\f12+\varepsilon}|t|^{-\f12+\varepsilon}x^{-\varepsilon}\ dt\r)\\
=&\frac {Cx}{2\pi i}\int_{1+\varepsilon-iT}^{1+\varepsilon+iT}D^{-\f12}\l(\f{6^{3}\pi^{3}x}{D}\r)^{-s}\Gamma(3s-2)f\l(\f{3\pi s}2\r)\zeta_K(s)\ ds+O\l(D^{\f12+\varepsilon}T^{\f12+\varepsilon}x^{-\varepsilon}\r).
\intertext{Changing the variable $3s-2$ to $s$, we have}
&\frac1{2\pi i}\int_{C_3}\zeta_K(s)\frac{x^s}s\ ds\\
=&\frac {Cx^{\f13}}{2\pi i}\int_{1+3\varepsilon-3iT}^{1+3\varepsilon+3iT}D^{\f16}\l(6\pi \l(\f xD\r)^{\f13}\r)^{-s}\Gamma(s)f\l(\f{\pi s}2+\pi\r)\zeta_K\l(\f13s+\f23\r)\ ds\\
&+O\l(D^{\f12+\varepsilon}T^{\f12+\varepsilon}x^{-\varepsilon}\r).\\
\end{align*}
Let $a(n)$ be the number of ideal of $\mathcal O_K$ with their ideal norm equal to $n$. Then the Dedekind zeta function $\zeta_K(s)$ can be expressed as 
\be\label{zeta}
\zeta_K(s)=\sum_{n=1}^{\infty}\f{a(n)}{n^s}\ \text{ for } \Re s>1.
\ee
The Dirichlet series (\ref{zeta}) is absolutely and uniformly convergent on compact subsets on $\Re(s)>1$. Therefore we can interchange the order of summation and integral. Thus we obtain
\ba
&\frac1{2\pi i}\int_{C_3}\zeta_K(s)\frac{x^s}s\ ds\\
=&\frac {Cx^{\f13}D^{\f16}}{2\pi i}\int_{1+3\varepsilon-3iT}^{1+3\varepsilon+3iT}\sum_{n=1}^{\infty}\f{a(n)}{n^{\f13s+\f23}}\l(6\pi \l(\f {x}D\r)^{\f13}\r)^{-s}\Gamma(s)f\l(\f{\pi s}2+\pi\r)\ ds\\
&+O\l(D^{\f12+\varepsilon}T^{\f12+\varepsilon}x^{-\varepsilon}\r)\\
=&\frac {Cx^{\f13}D^{\f16}}{2\pi i}\sum_{n=1}^{\infty}\f{a(n)}{n^{\f23}}\int_{1+3\varepsilon-3iT}^{1+3\varepsilon+3iT}\l(6\pi \l(\f {nx}D\r)^{\f13}\r)^{-s}\Gamma(s)f\l(\f{\pi s}2\r)\ ds\\
&+O\l(D^{\f12+\varepsilon}T^{\f12+\varepsilon}x^{-\varepsilon}\r).\\
\inter{Let $I_n$ the integral in above sum. Atkinson considered a sum similar to the above sum for $f=\cos$ \cite{at41}. He only used the Mellin transform and $\cos s=O(e^{|t|})$, so his result is not affected by replacing $\cos$ by $\sin$. Following his way, we estimate the integral in above sum. When $n\le X:=x^{3\alpha-1}$ for $\f12<\alpha<\f23$, }
I_n=&\f1{2\pi i}\int_{\f12-i\infty}^{\f12+i\infty}\l(6\pi \l(\f {nx}D\r)^{\f13}\r)^{-s}\Gamma(s)f\l(\f{\pi s}2\r)\ ds\\
&-\l(\int_{\f12-i\infty}^{\f12-3iT}+\int^{1+3\varepsilon-3iT}_{\f12-3iT}+\int_{\f12+3iT}^{\f12+i\infty}+\int_{1+3\varepsilon+3iT}^{\f12+3iT}\r)\l(6\pi \l(\f {nx}D\r)^{\f13}\r)^{-s}\Gamma(s)f\l(\f{\pi s}2\r)\ ds.\\
\inter{Since the Mellin transform of $f(x)$ is $f\l(\f{\pi s}2\r)\Gamma(s)$ for the function $f\in\{\sin, \cos\}$, we obtain  \be\label{mellin}f\l(6\pi \l(\f {nx}D\r)^{\f13}\r)=\f1{2\pi i}\int_{\f12-i\infty}^{\f12+i\infty}\l(6\pi \l(\f {nx}D\r)^{\f13}\r)^{-s}\Gamma(s)f\l(\f{\pi s}2\r)\ ds
\ee
from the Mellin inverse transform. And we deal with the other integrals by using the Stirling formula (\ref{stir}). We obtain \be
\int_{\f12\pm i\infty}^{\f12\pm3iT}\l(6\pi \l(\f {nx}D\r)^{\f13}\r)^{-s}\Gamma(s)f\l(\f{\pi s}2\r)\ ds=O\l(T^{\f12}\l(\f{D}{nx}\r)^{\f16}\r),\label{one}\ee
\be
\int^{1+3\varepsilon\pm3iT}_{\f12\pm3iT}\l(6\pi \l(\f {nx}D\r)^{\f13}\r)^{-s}\Gamma(s)f\l(\f{\pi s}2\r)\ ds=O\l(T^{\f12+\varepsilon}\l(\f{D}{nx}\r)^{\f13+\varepsilon}\r).\label{two}
\ee
Combine above two estimate (\ref{one}) and (\ref{two}) with identity (\ref{mellin}), it holds that when $n\le X$}
I_n=&f\l(6\pi \l(\f {nx}D\r)^{\f13}\r)+O\l(T^{\f12+\varepsilon}\l(\f{D}{nx}\r)^{\f13+\varepsilon}\r).\\
\inter{Next we deal with the case that $n\ge X$. From a way similar to the case that $n\le X$}
I_n=&\f1{2\pi i}\l(\int_{1+3\varepsilon-3iT}^{1+3\varepsilon-i}+\int^{1+3\varepsilon+i}_{1+3\varepsilon-i}+\int_{1+3\varepsilon+i}^{1+3\varepsilon+3iT}\r)\l(6\pi \l(\f {nx}D\r)^{\f13}\r)^{-s}\Gamma(s)f\l(\f{\pi s}2\r)\ ds\\
=&O\l(T^{1+\varepsilon}\l(\f{D}{nx}\r)^{\f13+\varepsilon}\r).\\
\end{align*}
\ba
\inter{Combine above results about $I_n$, }
&\frac1{2\pi i}\int_{C_3}\zeta_K(s)\frac{x^s}s\ ds\\
&=Cx^{\f13}D^{\f16}\sum_{n\le X}\f{a(n)}{n^{\f23}}f\l(6\pi \l(\f {nx}D\r)^{\f13}\r)+O\l(x^{-\varepsilon}D^{\f12+\varepsilon}T^{\f12+\varepsilon}\sum_{n\le X}\f{a(n)}{n^{1+\varepsilon}}\r)\\
&+O\l(x^{-\varepsilon}D^{\f12+\varepsilon}T^{1+\varepsilon}\sum_{n\ge X}\f{a(n)}{n^{1+\varepsilon}}\r)+O\l(D^{\f12+\varepsilon}T^{\f12+\varepsilon}x^{-\varepsilon}\r).\\
\end{align*}
Since the Dirichlet series $\sum\f{a(n)}{n^{1+\varepsilon}}$ is converges and it holds that 
\ba
\sum_{n\ge X}\f{a(n)}{n^{1+\varepsilon}}&=\int_{X}^{\infty}\f1{t^{1+\varepsilon}}\\
&=O\l(X^{-2-\varepsilon}\r)\\
&=O\l(x^{-6\alpha+2-\varepsilon}\r).
\end{align*}
This estimate gives 
\begin{eqnarray*}
&&\frac1{2\pi i}\int_{C_3}\zeta_K(s)\frac{x^s}s\ ds\\
&=&Cx^{\f13}D^{\f16}\sum_{n\le X}\f{a(n)}{n^{\f23}}f\l(6\pi \l(\f {nx}D\r)^{\f13}\r)+O\l(x^{2-6\alpha-\varepsilon}D^{\f12+\varepsilon}T^{1+\varepsilon}\r)+O\l(\f{x^{1+\varepsilon}D^{\f12+\varepsilon}}T\r).
\end{eqnarray*}
Since $\f12<\alpha<\f23$ we estimate 
\be\label{estim}\frac1{2\pi i}\int_{C_3}\zeta_K(s)\frac{x^s}s\ ds=Cx^{\f13}D^{\f16}\sum_{n\le X}\f{a(n)}{n^{\f23}}f\l(6\pi \l(\f {nx}D\r)^{\f13}\r)+O\l(\f{x^{1+\varepsilon}D^{\f12+\varepsilon}}T\r).\ee

\ba
\inter{We denote by $S$ the above sum. G. Kolesnik estimated a sum similar to $S$ for $\alpha=\f{53}{96}$ and arithmetic function $a(n)$ satisfying $|a(n)|=O(n^{\varepsilon})$ \cite{ko79}. We consider the sum $S$ for $\alpha=\f{53}{96}$ with following his way and we obtain }
S=&\sum_{n\le x^{\f{21}{32}}}\f{a(n)}{n^{\f23}}f\l(6\pi \l(\f {nx}D\r)^{\f13}\r)\\
=&\sum_{n\le x^{\f{11}{32}}}\f{a(n)}{n^{\f23}}f\l(6\pi \l(\f {nx}D\r)^{\f13}\r)+\sum_{x^{\f{11}{32}}\le n\le x^{\f{21}{32}}}\f{a(n)}{n^{\f23}}f\l(6\pi \l(\f {nx}D\r)^{\f13}\r).\\
\end{align*}
\ba
\inter{Let $N\le N'\le 2N\le x^{\f{21}{32}}$ then we can estimate the partial sum of $S$ as}
\sum_{N\le n\le N'}\f{a(n)}{n^{\f23}}f\l(6\pi \l(\f {nx}D\r)^{\f13}\r)&=O\left(N^{-\f23}\sum_{N\le n\le N_1}a(n)\exp\l(6\pi \l(\f {nx}D\r)^{\f13}\r)\r),
\inter{where $N_1\le N'$. In \cite{mu88} M\"uller noted that Kolesnik's way can be applied to this case and estimate this sum. Applying his result, it holds that}
\sum_{N\le n\le N_1}a(n)\exp\l(6\pi \l(\f {nx}D\r)^{\f13}\r)&=\l\{\begin{array}{ll}
O\l(N\r)& \text{ if } N\le x^{\f{11}{32}},\\
O\l(\l(\f xD\r)^{\f{11}{96}+\varepsilon}N^{\f23}+\l(\f xD\r)^{\f{7}{60}}N^{\f{197}{300}}\r)& \text{ if } x^{\f{11}{32}}\le N.\\
\end{array}\r.\\
\end{align*}
As a result, we can estimate the sum $S$ as
\[\sum_{n\le x^{\f{21}{32}}}\f{a(n)}{n^{\f23}}f\l(6\pi \l(\f {nx}D\r)^{\f13}\r)=O\l(x^{\f{11}{96}+\varepsilon}\r).\]
This result and estimate (\ref{estim}) lead to
\be\label{c3}
\frac1{2\pi i}\int_{C_3}\zeta_K(s)\frac{x^s}s\ ds=O\l(x^{\f{43}{96}+\varepsilon}D^{\f13}\r).
\ee
In the same way to estimate in the proof of Theorem \ref{num} for other integrals, it is obtained that
\be\label{last}
I_K(x)=\frac1{2\pi i}\int_{C_1+C_2+C_4}\zeta_K(s)\frac{x^s}s\ ds+O\left(\frac{x^{1+\varepsilon}}{T}\right)+O\l(\frac{x^{1+\varepsilon}T^{\varepsilon}D^{\varepsilon}}{T}\right).\\
\ee

\ba
\inter{The Cauchy residue theorem and estimate (\ref{c3}) and (\ref{last}) lead to }
I_K(x)&=\frac1{2\pi i}\int_{C}\zeta_K(s)\frac{x^s}s\ ds+O\l(x^{\f{43}{96}+\varepsilon}D^{\f13+\varepsilon}\r)+O\l(\frac{x^{1+\varepsilon}T^{\varepsilon}D^{\varepsilon}}{T}\right)\\
&=cx+O\l(x^{\f{43}{96}+\varepsilon}D^{\f13+\varepsilon}\r)+O\l(\frac{x^{1+\varepsilon}T^{\varepsilon}D^{\varepsilon}}{T}\right).\\
\intertext{We select $T=D^{\f16+\varepsilon}x^{\f{53}{96}}$, this becomes}
I_K(x)&=c x+O\l(x^{\f{43}{96}+\varepsilon}D^{\f13}\r).
\inter{This proves this theorem.}
\end{align*}
\end{proof}

When we fix the cubic field $K$, then $D$ becomes constant and our theorem (Theorem \ref{cubic}) agrees to M\"uller's theorem (Estimate (\ref{cub})). Lemma \ref{main} and Theorem \ref{cubic} lead to 
\begin{thm}
\label{cubicord}
For every $\varepsilon>0$ the following estimate holds.
\[E_m^r(x,K)=\l\{\begin{array}{ll}
O\l(x^{\f1r\l(\frac {139}{96}+\varepsilon\r)}D^{\frac23-\frac {m-1}2}\right)& \text{ if } rm=2,\\
O\left(x^{m-\frac {53}{96}+\varepsilon}D^{\frac13-\frac{m-1}2}\r)& \text{ otherwise }
\end{array}\r.\]
as  $x^{\f{53}{32}}D^{\f32}\rightarrow\infty$,
where $K$ runs through all cubic extension fields with $x^{\f{53}{96}-\varepsilon} > D^{\f56}$.
\end{thm}
If $K$ runs through all cubic extension fields with $x^{\f{5}{32}} > D$, then the uniform upper bounds of $E_m^r(x,K)$ in this theorem is better than that under assuming the Lindel\"  of Hypothesis. 

\section{Conjecture}
Theorem \ref{cubicord} states good uniform upper bounds. We propose that for all number fields $K$ the best uniform upper bound of the error term is better than that on the assumption of the Lindel\"  of Hypothesis. Our conjecture is 
\begin{conj}
\label{con}
For every $\varepsilon> 0$, we have
\[E_m^r(x,K)=\left\{
\begin{array}{ll}
o\l(x^{\f1r(\f32+\varepsilon)}D^{2\varepsilon-\f{m-1}2}\r)&\text{ if } rm=2,\\
o\l(x^{m-\f12+\varepsilon}D^{\varepsilon-\f{m-1}2}\r)&\text{ otherwise }
\end{array}
\right.\]
as $xD^2\rightarrow\infty$, where $K$ runs through all number fields with $x^{1-2\varepsilon}>D^{1+2\varepsilon}$.
\end{conj}
From Lemma \ref{main} it is very important to obtain good uniform bound of the distribution of ideals of $\mathcal O_K$. This conjecture \ref{con} is equivalent to the following statement:

For every $\varepsilon> 0$, we have
\be\label{idealco}
I_K(x)=cx+o\l(x^{\f12}D^{\varepsilon}\r)
\ee
as $xD^2\rightarrow\infty$, where $K$ runs through all number fields with $x^{1-2\varepsilon}>D^{1+2\varepsilon}$.

One can check this out easily.
On the other hand, when number field $K$ is fixed the following omega estimates is obtained in \cite{gk05}: 

For fixed number field $K$ with $n=[K:\mathbf Q]\ge2$ 
\be\label{omega}
I_K(x)=cx+\Omega\l(x^{\f12-\f{1}{2n}}(\log x)^{\f12-\f{1}{2n}}(\log\log x)^{\kappa}(\log\log\log x)^{\lambda}\r),
\ee
where $\kappa$ and $\lambda$ are constants depending on $K$. To be more precise, let $K^{gal}$ be the Galois closure of $K/\mathbf Q$ then two constants $\kappa$ and $\lambda$ depend on two Galois group $Gal\l(K^{gal}/K\r)$ and $Gal\l(K^{gal}/\mathbf{Q}\r)$.

From this estimate (\ref{omega}), our conjecture may not give the best estimate for uniform upper bound of $\Delta_K(x)$. In the case that $K=\mathbf Q(\sqrt{-1})$, it is known that considering $I_K(x)$ is equivalent to the Gauss circle problem. The Gauss circle problem states that $I_K(x)=cx+O(x^{\f14+\varepsilon})$ for every $\varepsilon> 0$ so Estimate (\ref{omega}) may be the best lower bound of $\Delta_K(x)$ for $K=\mathbf Q(\sqrt{-1})$. 

This conjecture is very difficult even when $K$ is fixed. When $K$ is a fixed number field with $n=[K:\mathbf Q]$ it is shown that for all $\varepsilon>0$
\be\label{good}\Delta_K(x)=\l\{\begin{array}{ll}
O\l(x^{\f{23}{73}}\l(\log x\r)^{\f{315}{146}}\r)& \text{ if }n=2,\\
O\l(x^{\f{43}{96}+\varepsilon}\r)& \text{ if }n=3
\end{array}\r.
\ee
in \cite{HW01} and \cite{mu88} respectively. This estimates (\ref{good}) are better than that under the assumption of the Lindel\"  of Hypothesis.
On the other hand if $K$ is a fixed number field with $n=[K:\mathbf Q]\ge4$ the best upper bound of $\Delta_K(x)$ hitherto is 
\be\label{hith} \Delta_K(x)=\l\{\begin{array}{ll}
O\l(x^{\f{41}{72}+\varepsilon}\r)& \text{ if }n=4,\\
O\l(x^{1-\f4{2n+1}+\varepsilon}\r)& \text{ if }5\le n\le10,\\
O\l(x^{1-\f3{n+6}+\varepsilon}\r)& \text{ if }11\le n
\end{array}\r.
\ee
for all $\varepsilon>0$. The case $4\le n\le10$ is shown in \cite{bo15} and the other case $11\le n$ is \cite{La10}.
This estimate (\ref{hith}) is not better than that under the assumption of the Lindel\"  of Hypothesis. 

Two Galois groups $Gal\l(K^{gal}/K\r)$ and $Gal\l(K^{gal}/\mathbf{Q}\r)$ have many informations about distribution of ideals from the algebraic number theory. Thus let $G$ be a fixed group and $H$ be a fixed normal subgroup of $G$, it is very important to improve the upper bound of $\Delta_K(x)$ as $K$ runs through all extensions where two groups $Gal\l(K^{gal}/K\r)$ and $Gal\l(K^{gal}/\mathbf{Q}\r)$ are $H$ and $G$ respectively.

\end{document}